\documentclass[reqno]{amsart}

\usepackage{amsmath}
\usepackage{amsfonts}
\usepackage{amssymb,enumerate}
\usepackage{amsthm}
\usepackage[all]{xy}
\usepackage{rotating}
\usepackage{hyperref}
\usepackage{color}

\theoremstyle{plain}
\newtheorem{lem}{Lemma}[section]
\newtheorem{cor}[lem]{Corollary}
\newtheorem{prop}[lem]{Proposition}
\newtheorem{thm}[lem]{Theorem}

\newtheorem*{mthm*}{Main Theorem}

\theoremstyle{definition}

\newtheorem{ex}[lem]{Example}

\newtheorem{para}[lem]{}

\newtheorem*{convention*}{Convention}

\newcommand{\catd}{\mathcal{D}}

\newcommand{\pd}{\operatorname{pd}}	
\newcommand{\gdim}{\mathrm{G}\text{-}\!\dim}

\newcommand{\id}{\operatorname{id}}	
\newcommand{\fd}{\operatorname{fd}}

\newcommand{\cidim}{\mathrm{CI}\text{-}\!\dim}

\newcommand{\depth}{\operatorname{depth}}	
\newcommand{\rank}{\operatorname{rank}}	

\newcommand{\edim}{\operatorname{edim}}

\newcommand{\len}{\operatorname{len}}

\newcommand{\cx}{\operatorname{cx}}

\newcommand{\rhom}{\mathbf{R}\!\operatorname{Hom}}	
\newcommand{\lotimes}{\otimes^{\mathbf{L}}}
\newcommand{\HH}{\operatorname{H}}

\newcommand{\s}{\mathfrak{S}}

\newcommand{\im}{\operatorname{Im}}
\newcommand{\shift}{\mathsf{\Sigma}}

\newcommand{\card}{\operatorname{card}}

\newcommand{\Ker}{\operatorname{Ker}}

\newcommand{\ideal}[1]{\mathfrak{#1}}
\newcommand{\m}{\ideal{m}}

\newcommand{\fm}{\ideal{m}}

\newcommand{\ti}{\tilde}
\newcommand{\comp}[1]{\widehat{#1}}
\newcommand{\ol}{\overline}
\newcommand{\wti}{\widetilde}

\newcommand{\bbz}{\mathbb{Z}}

\newcommand{\bbc}{\mathbb{C}}

\newcommand{\from}{\leftarrow}
\newcommand{\xra}{\xrightarrow}
\newcommand{\xla}{\xleftarrow}

\newcommand{\vf}{\varphi}

\newcommand{\y}{\mathbf{y}}

\newcommand{\x}{\mathbf{x}}

\renewcommand{\geq}{\geqslant}
\renewcommand{\leq}{\leqslant}
\renewcommand{\ker}{\Ker}

\newcommand{\Ext}[4][R]{\operatorname{Ext}_{#1}^{#2}(#3,#4)}	
\newcommand{\Rhom}[3][R]{\mathbf{R}\!\operatorname{Hom}_{#1}(#2,#3)}

\newcommand{\Hom}{\operatorname{Hom}}	
\newcommand{\Tor}[4][R]{\operatorname{Tor}^{#1}_{#2}(#3,#4)}

\newcommand{\ssm}{\smallsetminus}

\def\Tor{\operatorname{Tor}}
\def\Ext{\operatorname{Ext}}

\def\m{\mathfrak{m}}

\def\edim{\operatorname{edim}}

\numberwithin{equation}{lem}

\begin{document}

\bibliographystyle{amsplain}

\title[]{Ring homomorphisms and local rings with quasi-decomposable maximal ideal}

\author[S. Nasseh]{Saeed Nasseh}

\address{Department of Mathematical Sciences,
Georgia Southern University,
Statesboro, Georgia 30460, USA}

\email{snasseh@georgiasouthern.edu}
\urladdr{https://cosm.georgiasouthern.edu/math/saeed.nasseh}

\author{Keri Ann Sather-Wagstaff}
\address{School of Mathematical and Statistical Sciences,
Clemson University,
O-110 Martin Hall, Box 340975, Clemson, S.C. 29634
USA}
\email{ssather@clemson.edu}
\urladdr{https://ssather.people.clemson.edu/}

\author[R. Takahashi]{Ryo Takahashi}
\address{Graduate School of Mathematics\\
Nagoya University\\
Furocho, Chikusaku, Nagoya, Aichi 464-8602, Japan}
\email{takahashi@math.nagoya-u.ac.jp}
\urladdr{http://www.math.nagoya-u.ac.jp/~takahashi/}

\thanks{K. A. Sather-Wagstaff was supported by the National Science Foundation under Grant No.\ EES-2243134.
Any opinions, findings, and conclusions or recommendations expressed in this material are those of the author and do not necessarily reflect the views of the National Science Foundation. R. Takahashi was partly supported by JSPS Grants-in-Aid for Scientific Research 23K03070.}



\keywords{Cohen-Macaulay, complete intersection, decomposable maximal ideal, dualizing complex, Ext-friendly, fiber product, Gorenstein, hypersurface, injective dimension, local ring homomorphism, projective dimension, quasi-decomposable maximal ideal, semidualizing complexes, Tor-friendly, totally reflexive}
\subjclass[2020]{13D02, 13D05, 13D07, 13H10, 18A30}

\begin{abstract}
The notion of local rings with quasi-decomposable maximal ideal was formally introduced by Nasseh and Takahashi.
In separate works, the authors of the present paper show that such rings have rigid homological properties; for instance, they are both Ext- and Tor-friendly.
One point of this paper is to further explore the homological properties of these rings and also introduce new classes of such rings from a combinatorial point of view. 
Another point is to investigate how far some of these homological properties can be pushed along certain diagrams of local ring homomorphisms.
\end{abstract}

\maketitle


\section{Introduction}\label{sec161010a}
\begin{convention*}
Throughout the paper, $(R,\m_R,k)$ is a commutative noetherian local ring and $\comp R$ denotes the completion of $R$ in the $\m_R$-adic topology. If $R=\comp R$, then we say that $R$ is complete. By a ``fiber product ring'' we mean a fiber product of the form $S\times_kT$, where $S$ and $T$ are commutative noetherian local rings with a common residue field $k$ such that $S\neq k\neq T$; see~\ref{defn20230805j} for the definition and notation.
\end{convention*}

Ogoma~\cite{ogoma:edc} observed that the class of local rings with decomposable maximal ideal coincides with that of the fiber product rings; see~\ref{para20230518a} for details. The history of such rings goes back quite far because of their interesting properties and numerous applications; see, for instance, the works of Kostrikin and Šafarevič~\cite{kostrikin}, Dress and Krämer~\cite{dress}, Lescot~\cite{lescot:sbpfal}, Ogoma~\cite{ogoma:fpnra, ogoma}, and also the work of the authors and VandeBogert~\cite{nasseh:ahplrdmi}.
In recent years, further progress has been made on the structure and homological properties of these rings, as we explain after the next paragraph.\vspace{4pt}

Following~\cite{avramov:phcnr}, the local ring $R$ is called \emph{Ext-friendly} (resp. \emph{Tor-friendly}) if for every pair $(M,N)$ of finitely generated $R$-modules, the condition $\Ext^i_R(M,N)=0$ (resp. $\Tor^R_i(M,N)=0$) for $i\gg 0$ implies that $\pd_R(M)<\infty$ or $\id_R(N)<\infty$ (resp. $\pd_R(M)<\infty$ or $\pd_R(N)<\infty$).
By~\cite[Propositions 2.2 and 5.5]{avramov:phcnr}, Tor-friendliness implies Ext-friendliness. Also, an Ext-friendly ring $R$ satisfies the \emph{Auslander-Reiten Conjecture} that states if $\Ext^i_R(M,M\oplus R)=0$ for a finitely generated $R$-module $M$ and all $i\geq 1$, then $M$ is free; see~\cite{AR} for the history of this conjecture.\vspace{4pt}

In~\cite[Theorem 3.1]{nasseh:oeire}, Nasseh and Yoshino showed that a fiber product ring of the form $S\times_k\left(k[x]/(x^2)\right)$ is Tor-friendly (hence, Ext-friendly). A few years later, Nasseh and Sather-Wagstaff~\cite[Theorem 1.1]{nasseh:vetfp} generalized this result by proving that any fiber product ring of the form $S\times_kT$ is Tor-friendly.\footnote{Another generalization of the result of Nasseh and Yoshino~\cite[Theorem 3.1]{nasseh:oeire} to the differential graded homological algebra setting is found in~\cite[Theorem 4.1]{AINSW}.} Furthermore, Nasseh and Takahashi~\cite[Theorem A]{nasseh:lrqdmi} proved that the maximal ideal of a fiber product ring is always a direct summand of a direct sum of certain syzygies of finitely generated modules of infinite projective dimension. Several other properties and applications of these rings have also been studied in~\cite{Ananth, christensen:gmirlr, Goto, WW, Hugh, Iarrobino, moore, NT, NV, takahashi:dssmrcf}.\vspace{4pt}

Although, fiber product rings (or, local rings with decomposable maximal ideal) have nice properties and applications, there are two particular vexing facts about them. First of all, these rings are not integral domain; see our discussion in~\ref{para20230817a}. Second of all, by~\cite{lescot:sbpfal}, depth of such rings is always $\leq 1$, while their Krull dimension can be any positive integer. Therefore, a randomly given fiber product ring is most likely non-Cohen-Macaulay; see also~\cite[Fact 2.2]{nasseh:ahplrdmi}. These facts motivated Nasseh and Takahashi~\cite{nasseh:lrqdmi} to consider a more general version of such rings, namely, the class of local rings that are \emph{deformations} of fiber product rings. Such rings are called local rings with \emph{quasi-decomposable} maximal ideal.\vspace{4pt}

Several classes of Cohen-Macaulay and non-Gorenstein local rings with quasi-decomposable maximal ideal that are integral domain have been introduced in~\cite{nasseh:lrqdmi}. Such classes include certain numerical semigroup rings as well as Cohen-Macaulay singular local rings with infinite residue field and minimal multiplicity (e.g., $2$-dimensional non-Gorenstein normal local domains with a rational
singularity); see Example~\ref{ex20230826a} for more details. In Section~\ref{sec171118b}, we prove the following result that introduces new classes of both Cohen-Macaulay and non-Cohen-Macaulay local rings with
quasi-decomposable maximal ideal from a combinatorial point of view. (See~\ref{defn171101a} for the terminology.)
\begin{thm}\label{thm20230825a}
Let $G$ be a finite simple graph on $n$ vertices with $v_n$ a star vertex.
\begin{enumerate}[\rm(a)]
\item\label{prop171124a}
The complete local ring $k[\![\shift G]\!]$ over the field $k$ is Cohen-Macaulay of dimension $n$ with quasi-decomposable maximal ideal.
\item\label{prop171124b}
Let $\wti G$ be obtained from
$G$ by adding a whisker to  each vertex except for $v_n$.
Then, the complete local ring $k[\![\wti G]\!]$ over the field $k$ has dimension $n$, depth $n-1$, and quasi-decomposable maximal ideal.
\end{enumerate}
\end{thm}

As one might expect, local rings with quasi-decomposable maximal ideal have
rigid homological properties like those of the fiber product rings. For instance, these rings are Tor-friendly (hence, Ext-friendly) by~\cite[Corollaries 6.5 and 6.8]{nasseh:lrqdmi}; see also~\cite{ryo}, where Takahashi studies these rings as a special case of, so-called, dominant local rings.
One point of the present paper is to further explore the homological properties of such rings. Therefore, from this point of view, a part of this paper can be considered as an addendum to~\cite{nasseh:vetfp, nasseh:lrqdmi}. For instance, among other results in this direction we prove the following theorem in Section~\ref{sec161010b} which is a generalization of~\cite[Corollary 2.7 and Fact 2.9]{nasseh:ahplrdmi}.

\begin{thm}\label{cor161010bz}
If $\fm_R$ is quasi-decomposable, then $R$ is Gorenstein if and only if it is a hypersurface.
Moreover, if these equivalent conditions are satisfied, then $\dim(R)\geq 1$ and $e(R)\leq 2$. (Here, $e(R)$ is the Hilber-Samuel multiplicity of $R$.)
\end{thm}

Another point of the present paper is as follows: local rings which are somewhat similar can be distinguished by the property of having quasi-decomposable maximal ideal (or not); see, for instance, Examples~\ref{ex20220511a} and~\ref{ex20220511b} and their subsequent paragraph. This persuades us to consider a relaxed version of the quasi-decomposable maximal ideal condition in some results of this paper. More precisely, we will
investigate how far we can push some of our results along certain diagrams of local ring homomorphisms starting with any local ring and ending with a local ring that has quasi-decomposable maximal ideal. The first result of this series is the following that generalizes part of Theorem~\ref{cor161010bz} and is proven in Section~\ref{sec161010b}.

\begin{thm}\label{prop171118a}
Assume $R$ is Gorenstein.
The following conditions are equivalent.
\begin{enumerate}[\rm(i)]
\item \label{prop171118a1}
$R$ is a complete intersection.
\item \label{prop171118a2}
$R$ admits a diagram of deformations
$R\to R'\from R''$ such that $R''$ has quasi-decomposable maximal ideal.
\item \label{prop171118a3}
$\comp R$ admits a deformation $\comp R\from Q$, where $Q$ has quasi-decomposable maximal ideal.
\item \label{prop171118a4}
There exists a quasi-Gorenstein local ring homomorphism $R\to\ti R$ of finite complete intersection dimension such that
$\ti R$ admits a finite sequence
$$
\ti R\from R_1\to R_2\from\cdots\to R_n
$$
of deformations in which $R_n$ has decomposable maximal ideal.
\item \label{prop171118a5}
There exists a quasi-Gorenstein local ring homomorphism $R\to\ti R$ of finite complete intersection dimension such that
$\ti R$ admits a finite sequence
\begin{equation}\label{eq20230519a}
\ti R\from R_1\to R_2\from\cdots\to R_n
\end{equation}
of complete intersection local ring homomorphisms in which $R_n$ has quasi-decomposable maximal ideal.
\end{enumerate}
\end{thm}

A generalized version of a conjecture of Tachikawa~\cite[Chapter 8]{Tachikawa} in commutative algebra, which is a special case of the Auslander-Reiten Conjecture, states that if $R$ is Cohen-Macaulay with a canonical module $\omega$, then $\Ext^i_R(\omega,R)=0$ for all $i\geq 1$ implies that $R$ is Gorenstein; see~\cite{ABS} for the history of this conjecture.
In the following result, which is a souped up version of Proposition~\ref{prop171123b}, we study Tachikawa's Conjecture along a diagrams of local ring homomorphisms.

\begin{thm}\label{prop171123c}
A singular local ring $R$ is Gorenstein if any of the following holds.
\begin{enumerate}[\rm(a)]
\item\label{prop171123c1}
There exists
a non-zero finitely generated $R$-module $N$ with $\id_R(N)<\infty$ such that $\Ext^{i}_R(N,R)=0$ for  $i\gg 0$,
and there is a diagram of local ring homomorphisms $R\xra\vf R'\xla\psi S$ such that
$\vf$ is a composition of flat local maps and deformations,
$\psi$ is a deformation,
and $S$ has quasi-decomposable maximal ideal.
\item\label{prop171123c2}
There exists a finitely generated $R$-module $M$ with $\pd_R(M)=\infty$ such that
$\Ext^{i}_R(M,R)=0$ for  $i\gg 0$,
and there is a local ring homomorphism $R\xra\vf S$  that
is a composition of flat local maps and deformations
and $S$ has quasi-decomposable maximal ideal.
\end{enumerate}
\end{thm}

Finally, following the theme of Theorems~\ref{prop171118a} and~\ref{prop171123c}, our goal in Section~\ref{sec170206a} is to study the cardinality of the set $\s(R)$ that consists of shift-isomorphism classes of semidualizing $R$-complexes along diagrams of local ring homomorphisms. The set $\s(R)$ is known to be a finite set that, in general, can be big.
However, our main result in Section~\ref{sec170206a}, stated next, shows that under the existence of certain diagrams of local ring homomorphisms this set is small.

\begin{thm}\label{cor170213a}
Assume that $R$ admits a diagram of local ring homomorphisms
$$R=R_0\to R_1\from R_2\to\cdots\from R_n$$
such that $R_n$ has quasi-decomposable maximal ideal.
Assume that each leftward pointing map
is complete intersection such that the induced map on residue fields is an isomorphism.
Assume further that each rightward pointing map has finite complete intersection dimension.
Then $\card(\s(R))\leq 2$.
\end{thm}

\section{Local ring homomorphisms: general background}\label{sec171120a}

\begin{para}
Throughout this paper, $\catd(R)$ denotes the derived category of $R$, where the objects are the (possibly unbounded) $R$-complexes. An $R$-complex $X$ is called \emph{homologically bounded} if $\HH_i(X)=0$ for $|i|\gg 0$. 
An $R$-complex $X$ is \emph{homologically finite} if it is homologically bounded and each $\HH_i(X)$ is finitely generated.
The right and left derived functors of Hom and tensor product functors in $\catd(R)$ are denoted by $\rhom_R(-,-)$ and $-\lotimes_R-$, respectively.
For an integer $i$, the \emph{$i$-th shift} of an $R$-complex $X$ is denoted by $\shift^i X$. Note that $\left(\shift^i X\right)_j = X_{j-i}$ with $\partial_j^{\shift^i X}=(-1)^i\partial_{j-i}^X$ for all integers $j$. 
Quasiisomorphisms of $R$-complexes, i.e., isomorphisms in $\catd(R)$, are denoted by the symbol $\simeq$.
\end{para}

\begin{para}
We say that $R$ is a \emph{deformation of $S$} if there is a surjective ring homomorphism $\varphi\colon S\to R$ with $\ker(\varphi)$ generated by an $S$-regular sequence. In this case, we may also say that $\varphi$ is a deformation. The minimal number of generators of $\ker(\varphi)$ is called the \emph{codimension of $\varphi$}.
\end{para}

\begin{para}\label{defn171118a}
We say that $R$ is a \emph{complete intersection} if there is a deformation $\psi\colon S\to \comp R$, where $S$ is a regular local ring.
If $\ker(\psi)$ is principal, then $R$ is called a \emph{hypersurface}.
(In particular, we take the perspective that a regular local ring is a hypersurface.)
\end{para}


\begin{para}[\cite{AFH}]\label{defn20230805t}
Let $\vf\colon R\to S$ be a local ring homomorphism. We denote by $\grave\vf\colon R\to \comp S$ the composition of $\vf$ with the natural map $S\hookrightarrow \comp S$.

A \emph{Cohen factorization} of $\vf$ is a diagram $R\xra{\dot{\varphi}} R'\xra{\varphi'} S$ of local ring homomorphisms
such that $R'$ is complete, $\varphi=\varphi'\dot{\varphi}$, the map $\dot{\varphi}$ is 
flat with regular closed fibre, and $\varphi'$ is surjective. 
If $S$ is complete, then it follows from~\cite[(1.1) Theorem and (1.5) Proposition]{AFH} that $\varphi$ has a Cohen factorization.
\end{para}

\begin{para}[\cite{avramov:lcih}]\label{defn20230805e}
Let $\vf\colon R\to (S,\m_S)$ be a local ring homomorphism and $R\xra{\dot\vf} R'\xra{\vf'} \comp S$ be a Cohen factorization of $\grave\vf$. We say that
$\vf$ is \emph{complete intersection at $\m_S$} (or simply complete intersection) if
$\vf'$ is a deformation. Note that this definition is independent of the choice of Cohen factorization; see~\cite[(3.3) Remark]{avramov:lcih}. Also $R$ and $\vf$ are complete intersection if and only if
$S$ is complete intersection and $\fd_R(S)<\infty$;
see~\cite[(5.9), (5.10), and (5.12)]{avramov:lcih}.
\end{para}

\begin{para}[\cite{auslander:smt}]
We say that a finitely generated $R$-module $L$ has \emph{Gorenstein dimension $0$}, and write $\gdim_R(L)=0$, if the following conditions are satisfied:
\begin{enumerate}[\rm(i)]
\item
the canonical map $L\to L^{**}$ is an isomorphism, where $(-)^*=\Hom_R(-,R)$;
\item
$\Ext^i_R(L,R)=0=\Ext^i_R(L^*,R)$ for all $i\geq 1$.
\end{enumerate}
Modules with Gorenstein dimension $0$ are also called \emph{totally reflexive}.

For a non-negative integer $n$, we say that a finitely generated $R$-module $M$ has \emph{Gorenstein dimension at most $n$}, and write $\gdim_R(M)\leq n$, if there exists an exact sequence $0\to L_n\to \cdots \to L_1\to L_0\to M\to 0$ of finitely generated $R$-module such that $\gdim_R(L_i)=0$ for all $0\leq i\leq n$. If such an exact sequence does not exist, then we say $M$ has infinite Gorenstein dimension, and write $\gdim_R(M)=\infty$.

If $R$ is Gorenstein, then for every finitely generated $R$-module $M$ we have $\gdim_R(M)<\infty$. Conversely, if $\gdim_R(k)<\infty$, then $R$ is Gorenstein; see~\cite{auslander:smt}.
\end{para}

\begin{para}[\cite{avramov:rhafgd}]\label{para20230806n}
Using the notation from~\ref{defn20230805e}, we set
$$\gdim(\vf):=
\gdim_{R'}(\comp S)-\edim(\dot\vf)
$$
where $\edim(\dot\vf)$ denotes the embedding dimension of the regular closed fibre of $\dot\vf$. Note that, by~\cite[3.2. Theorem]{iyengar:golh}, this definition is independent of the choice
of Cohen factorization. Moreover, it follows from the definition that if $\varphi$ is complete intersection, then $\gdim(\vf)<\infty$. 
\end{para}

\begin{para}
Let $X$ be a homologically finite $R$-complex. The \emph{Poincar\'e} and \emph{Bass series} of $X$, denoted $P_X^R(t)$ and $I^X_R(t)$, respectively,
are the formal power series
$$
P_X^R(t):=\sum_{i\geq 0}\rank_{k}(\Tor^R_i(X,k))t^i\qquad\text{and}\qquad
I^X_R(t):=\sum_{i\geq 0}\rank_{k}(\Ext_R^i(k,X))t^i.
$$
\end{para}

\begin{para}[\protect{\cite[(7.1) Theorem]{avramov:rhafgd}}]\label{para20230806a}
Let $\vf\colon R\to (S, \m_S)$ be a local ring homomorphism with $\gdim(\vf)<\infty$.
The \emph{Bass series of $\vf$}, denoted $I_{\vf}(t)$, is a formal Laurent series
with non-negative integer coefficients satisfying the formal relation
\begin{equation}\label{eq20230811a}
I^S_S(t)=I^R_R(t)I_{\vf}(t).
\end{equation}
We say that $\vf$ is \emph{quasi-Gorenstein at $\m_S$} (or simply quasi-Gorenstein) if $I_{\vf}(t)=t^a$ for some integer $a$. In this case, it follows from~\cite[(7.4) Theorem]{avramov:rhafgd} that $a=\depth(S)-\depth(R)$. 
By~\cite[(7.7.2)]{avramov:rhafgd}, the ring $S$ is Gorenstein if and only if $R$ is Gorenstein and $\vf$ is
quasi-Gorenstein. 

If $\vf$ is quasi-Gorenstein and $\fd_R(S)<\infty$, then $\varphi$ is called \emph{Gorenstein at $\m_S$} (or simply Gorenstein). By~\cite[(7.2) Theorem]{avramov:glh} we have that $R$ and $\vf$ are Gorenstein if and only if
$S$ is Gorenstein and $\fd_R(S)<\infty$.
\end{para}

\begin{para}[\cite{avramov:cid}]
A diagram $R\xra{\vf} R'\xla{\pi} S$
of local ring homomorphisms is called a 
\emph{quasi-deformation} if $\vf$ is
flat and $\pi$ is a deformation.
The \emph{complete intersection dimension of an $R$-module $M$}, denoted $\cidim_R(M)$, is defined to be
$$
\!\cidim_R(M)\!:=\inf\{\pd_S(M\otimes_{R}R')-\pd_S(R')\mid R\to R'\leftarrow S\ \text{is a quasi-deformation}\}.
$$
If $R$ is complete intersection, then for every finitely generated $R$-module $M$ we have $\cidim_R(M)<\infty$. Conversely, if $\cidim_R(k)<\infty$, then $R$ is complete intersection; see~\cite[(1.3) Theorem]{avramov:cid}. 
\end{para}

\begin{para}[\cite{sather:cidfc}]
Let $\vf\colon R\to S$ be a local ring homomorphism.
The \emph{complete intersection dimension of  $\vf$}, denoted $\cidim(\vf)$,
is defined to be
$$
\cidim(\vf):= 
\inf\left\{\cidim_{R'}(\comp{S})-\edim(\Dot{\vf})
\left| \text{\begin{tabular}{c}
$R\xra{\dot\vf}R'\xra{\vf'}\comp{S}$ is a Cohen \\
factorization of $\grave\vf$
\end{tabular}}\right.\!\!\!\right\}.
$$
It is unknown whether the finiteness of $\cidim(\vf)$ 
is independent of the choice
of Cohen factorization.
\end{para}

\begin{para}\label{defn171120f}
If $\vf\colon R\to S$ is a local ring homomorphism with $\cidim(\vf)<\infty$ and $S$ is a complete intersection, then $R$ is a complete intersection.
Indeed, use a Cohen factorization to reduce to the case where $\vf$ is surjective. In this case, we have $\cidim_R(S)<\infty$, and therefore,
$\cx_R(S)<\infty$; see~\cite[(5.6) Theorem]{avramov:cid}. (Here, $\cx_R(S)$ denotes the complexity of $S$ over $R$; see~\cite{avramov:ifr} for the definition.) If $S$ is a complete intersection, then by~\cite[Theorem 8.1.2]{avramov:ifr} we have $\cx_S(k)<\infty$. It then follows from~\cite[Theorem 9.1.1(1) and Remark 7.1.1]{avramov:holh} that
$\cx_R(k)<\infty$. Thus, again by~\cite[Theorem 8.1.2]{avramov:ifr} we conclude that $R$ is a complete intersection.
\end{para}

The next discussion uses the notion of (semi)dualizing complexes. For the definitions of these complexes and more we refer the reader to Section~\ref{sec170206a}.

\begin{para}[\cite{avramov:rhafgd}]\label{para20230808d}
Let $\vf\colon R\to S$ be a local ring homomorphism, and let $D^{\comp R}$ be a dualizing $\comp R$-complex.
A \emph{dualizing complex for $\vf$} is a semidualizing $S$-complex
$D^{\vf}$ with the property that $D^{\comp R}\lotimes_{\comp R}(\comp S\lotimes_S D^{\vf})$
is a dualizing $\comp S$-complex.
If we assume that $\gdim(\varphi)<\infty$, then
a dualizing complex $D^{\grave \vf}$ for 
$\grave\vf$ exists by~\cite[(6.7) Lemma]{avramov:rhafgd}.


\end{para}

\section{Local rings with quasi-decomposable maximal ideal}
\label{sec171118b}

This section is devoted to the definition of local rings with quasi-decomposable maximal ideal -- a notion that was formally introduced by Nasseh and Takahashi in~\cite{nasseh:lrqdmi} -- and to the proof of Theorem~\ref{thm20230825a} in which we introduce combinatorially constructed classes of such rings. The class of local rings with quasi-decomposable maximal ideal naturally includes that of local rings with decomposable maximal ideal. Therefore, we start this section with the following definition; see Remark~\ref{para20230518a}.

\begin{para}\label{defn20230805j}
Let $(S,\fm_S,k)$ and $(T,\fm_T,k)$ be commutative noetherian local rings. The \emph{fiber product} of $S$ and $T$ over their common residue field $k$ is defined to be
$$
S\times_k T:=\left\{(s,t)\in S\times T\mid \pi_S(s)=\pi_T(t)\right\}
$$
where $S\xra{\pi_S} k\xla{\pi_T}T$ are the natural surjections. Note that $S\times_k T$ is a local ring with maximal ideal $\fm_{S\times_k T}=\fm_S\oplus \fm_T$ and residue field $k$.

We say that the local ring $R$ is a (non-trivial) fiber product if there exist local rings $(S,\fm_S,k)$ and $(T,\fm_T,k)$ with $S\neq k\neq T$ such that $R\cong S\times_k T$.
\end{para}

\begin{para}\label{para20230518a}
It follows from~\cite[Lemma 3.1]{ogoma:edc} (or~\cite[Fact 3.1]{nasseh:lrqdmi}) that the class of
fiber product rings coincides with the class of local rings with decomposable maximal ideal. More precisely, if $\fm_R=I\oplus J$ is a non-trivial decomposition of $\fm_R$, then $R\cong R/I\times_kR/J$.
\end{para}

\begin{para}
One can check that for any field $k$ there is a ring isomorphism
$$
\frac{k[[x_1,\ldots,x_n]]}{(f_1,\ldots,f_u)}\times_k \frac{k[[y_1,\ldots,y_m]]}{(g_1,\ldots,g_v)}\cong\frac{k[[x_1,\ldots,x_n,y_1,\ldots,y_m]]}{\left(f_1,\ldots,f_u,g_1,\ldots,g_v,x_iy_j\mid \begin{matrix}1\leq i\leq n\\ 1\leq j\leq m\end{matrix}\right)}.
$$
\end{para}

\begin{para}\label{para20230812s}
The maximal ideal $\fm_R$ of the ring $R$ is called \emph{quasi-decomposable}
if there is an $R$-regular sequence $\x\in\m_R$ such that $\fm_R/(\x)$ is decomposable. In this case we say that $R$ has \emph{quasi-decomposable maximal ideal}.
By~\ref{para20230518a}, $R$ has quasi-decomposable maximal ideal if it is a deformation of a fiber product ring.
\end{para}

\begin{para}\label{para20230817a}
Identifying $\m_S$ and $\m_T$ with the ideals $\m_S\oplus 0$ and $0\oplus \m_T$ of $S\times_kT$ in~\ref{defn20230805j}, note that $\m_S\m_T=0$. Hence, fiber product rings (e.g., local rings with decomposable maximal ideal) are not integral domains, thus, not regular. However, the following result holds true; see~\ref{disc20230522a} for a more detailed discussion.
\end{para}
\begin{prop}\label{disc171118a}
If $R$ is a regular local ring of dimension $n\geq 2$, then $\fm_R$ is quasi-decomposable.
\end{prop}

\begin{proof}
First assume that $n=2$. Let $R'=R/(xy)$, where $x,y\in \m_R$ is a regular system of parameters.
We show that the maximal ideal $\m_{R'}=(x,y)R'$ is decomposable.
Since $R$ is a unique factorization domain, we have $xR\cap yR=xyR$.
Thus, $xR'\cap yR'=(0)$. This implies that $\m_{R'}=(x,y)R'=xR'\oplus yR'$, as desired.

Now we prove the general case where $n\geq 3$. Let $r_1,\ldots,r_n\in \m_R$ be a regular system of parameters and note that $\ol R=R/(r_3,\ldots,r_n)$ is a $2$-dimensional regular ring. Hence, by the previous case, $\ol R$ has quasi-decomposable maximal ideal.
Since $r_3,\ldots,r_n$ is $R$-regular, $R$ also
has quasi-decomposable maximal ideal.
\end{proof}

Several classes of local rings with quasi-decomposable maximal ideal (that are not fiber products) have been introduced in~\cite{nasseh:lrqdmi}. Such classes include the following.

\begin{ex}\label{ex20230826a}
The ring $R$ has quasi-decomposable in any of the following cases.
\begin{enumerate}[\rm(a)]
\item
$R$ is a singular Cohen-Macaulay local ring with infinite residue field and minimal multiplicity, e.g., $R$ is a 2-dimensional non-Gorenstein
normal local domain with a rational singularity.

\item
$R=k[[H]]$ is a local complete numerical semigroup ring over a field $k$, where $H=\langle pq+p+1, 2q+1, p+2\rangle$
is the numerical semigroup with $p,q>0$ and $\gcd(p + 2, 2q + 1) = 1$.

\item
$R=k[[H]]$ is a non-Gorenstein almost-Gorenstein numerical semigroup ring with $\edim(R)=3$ and $e(R)\leq 6$. (Here, $\edim(R)$ denoted the embedding dimension of $R$.)

\item
$R$ is any of the Cohen-Macaulay local rings in~\cite[Examples 7.1, 7.2, 7.4, 7.5]{crs}.
\end{enumerate}
\end{ex} 

In the rest of this section, we prove Theorem~\ref{thm20230825a} that introduces new classes of Cohen-Macaulay and non-Cohen-Macaulay local rings with quasi-decomposable maximal ideal from a combinatorial point of view. We assume that the reader is familiar with combinatorial aspects of commutative algebra. However, to avoid confusion, we specify some terminology.

\begin{para}\label{defn171101a}
Let $G$ be a finite simple graph (i.e, $G$ has no loops and no multiple edges) with vertex set $V=\{v_1,\ldots,v_n\}$ and edge set $E$.
Consider the polynomial ring $S=k[v_1,\ldots,v_n]$ over a field $k$.
The \emph{edge ideal} of $G$ in $S$, denoted $I(G)$, is the ideal generated by the edges of $G$, that is,
$$
I(G):=\left(\{v_iv_j\in S\mid v_iv_j\in E\}\right)S.
$$
Set $k[G]:=S/I(G)$ and $k[\![G]\!]:=\comp{k[G]}=k[\![v_1,\ldots,v_n]\!]/(I(G))$, where $\comp{k[G]}$ is the completion of $k[G]$ with respect to the graded maximal ideal of $k[G]$.

Let $W=\{w_1,\ldots,w_n\}$ denote a second list of vertices. By $\shift G$ we denote the graph obtained from $G$ by
adding a ``whisker'' at each vertex of $G$, that is, $\shift G$ has vertex set $V\cup W$ and edge set
$\{v_iw_i\mid i=1,\ldots,n\}\cup E$.
\end{para}

\begin{ex}
Here are a couple of examples of edge ideals and whiskered graphs.
$$
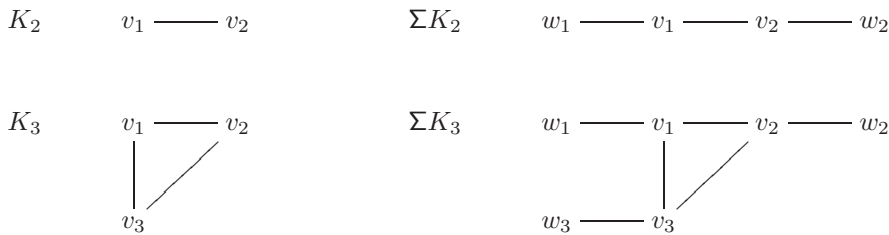
\xymatrix{
K_2 & v_1\ar@{-}[r] & v_2
&&\shift K_2 & w_1\ar@{-}[r] & v_1\ar@{-}[r] & v_2\ar@{-}[r] & w_2
\\
K_3 & v_1\ar@{-}[r]\ar@{-}[d] & v_2\ar@{-}[ld]
&&\shift K_3 & w_1\ar@{-}[r] & v_1\ar@{-}[r] \ar@{-}[d]& v_2\ar@{-}[r]\ar@{-}[ld] & w_2
\\
&v_3&&&&
w_3\ar@{-}[r]&v_3
}$$
\begin{align*}
I(K_2)
&=(v_1v_2)
&
I(\shift K_2)
&=(v_1v_2,v_1w_1,v_2w_2)
\\
I(K_3)
&=(v_1v_2,v_1v_3,v_2v_3)
&
I(\shift K_3)
&=(v_1v_2,v_1v_3,v_2v_3,v_1w_1,v_2w_2,v_3w_3)
\end{align*}
\end{ex}


Theorem~\ref{thm20230825a}(\ref{prop171124a}) follows directly from the next discussion.

\begin{para}
Continue with the terminology of~\ref{defn171101a}.
The edge ideal $I(\shift G)$ of the ring $S':=k[v_1,\ldots,v_n,w_1,\ldots,w_n]$
is Cohen-Macaulay, i.e., the quotient ring $k[\shift G]=S'/I(\shift G)$ is Cohen-Macaulay by~\cite[Proposition~2.2]{V1}
and~\cite[Proposition~6.3.2]{V2}.
Specifically, the ring $k[\shift G]$ has dimension $n$, and also the sequence $v_1-w_1,\ldots,v_n-w_n$ is $k[\shift G]$-regular with
$k[\shift G]/(v_1-w_1,\ldots,v_n-w_n)$ isomorphic to the local artinian ring
$k[G]':=k[v_1,\ldots,v_n]/(I(G),v_1^2,\ldots,v_n^2)$.
(One way to view this is via polarization of the non-square-free ideal $(I(G),v_1^2,\ldots,v_n^2)$; for a discussion on polarization see, for instance,~\cite{Sarah}.)
It follows that $v_1-w_1,\ldots,v_n-w_n$ is also $k[\![\shift G]\!]$-regular with
$$k[\![\shift G]\!]/(v_1-w_1,\ldots,v_n-w_n)\cong k[G]'.$$

From the definition of $k[G]'$, it is straightforward to show that
the socle elements of $k[G]'$ are in bijection with the maximal cliques
in the complementary graph $G^c$.
For instance, for the path $P_3=(v_1-v_2-v_3)$, the complementary graph consists of the edge $v_1-v_3$ and the isolated vertex $v_2$.
This gives two maximal cliques in $P_3^c$ (that are the connected components of $P_3^c$)
corresponding to the socle elements $v_1v_3$ and $v_2$ in $k[P_3]'=k[v_1,v_2,v_3]/(v_1^2,v_2^2,v_3^3,v_1v_2,v_2v_3)$.
Notice in this example that the vertex $v_2$ is a star-vertex, that is, it is adjacent to every other vertex in $P_3$.
Notice further that this element shows that $k[P_3]'$ is a fiber product as follows:
\begin{align*}
k[P_3]'
&=\frac{k[\![v_1,v_2,v_3]\!]}{(v_1^2,v_2^2,v_3^3,v_1v_2,v_2v_3)}
\cong \frac{k[\![v_2]\!]}{(v_2^2)}\times_k\frac{k[\![v_1,v_3]\!]}{(v_1^2,v_3^3)}.
\end{align*}
It follows that $k[\![\shift P_3]\!]$ has quasi-decomposable maximal ideal.

In general,
this process (with star vertex $v_n$) yields the isomorphism 
$$k[G]'\cong k[\![v_n]\!]/(v_n^2)\times_k k[H]'$$
where $H$ is the subgraph of $G$ induced
by the vertices $v_1,\ldots,v_{n-1}$.
On the other hand, if one only mods out by $v_2-w_2,\ldots,v_n-w_n$, then one obtains a quotient isomorphic to
the 1-dimensional Cohen-Macaulay fiber product
$k[\![v_n]\!]\times_k(k[H]'[\![w_n]\!])$;
see~\ref{disc171124a} for a concrete example.
\end{para}



\begin{para}\label{disc171124a}
For a list of variables $\underline{X}=X_{1,1},X_{2,1},\ldots,X_{1,n},X_{2,n}$, consider the ideal 
$$I=(X_{1,1},X_{2,1})^2+\cdots+(X_{1,n},X_{2,n})^2$$
of the ring $k[\![\underline{X}]\!]$ over a field $k$. For the new variables $Y,Z$, the ring
$$
R=\frac{k[\![Z,\underline X,Y]\!]}{(I,Z\underline{X},ZY)}
$$
constructed in~\cite[Proof 4.1]{nasseh:ahplrdmi} is a $1$-dimensional Cohen-Macaulay fiber product ring that arises from an edge ideal construction.
Specifically, consider the following graph $G$ obtained by connecting $n$ paths of length 1 to a single vertex $v$.
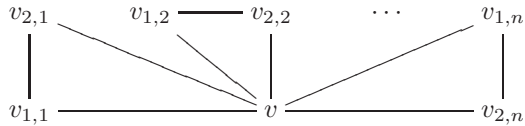
$$\xymatrix{
v_{2,1}\ar@{-}[rrd]
&v_{1,2}\ar@{-}[r]\ar@{-}[rd]
&v_{2,2}\ar@{-}[d]
&\cdots
&v_{1,n}\ar@{-}[d]\ar@{-}[lld]
\\
v_{1,1}\ar@{-}[u]\ar@{-}[rr]
&&v&&v_{2,n}\ar@{-}[ll]
}$$
Note that $v$ is a star vertex for this graph.
Thus, the ring $k[\![\shift G]\!]$ has quasi-decomposable maximal ideal by Theorem~\ref{prop171124a}.
In fact, modding $k[\![\shift G]\!]$ out by the regular sequence of elements of the form $v_{i,j}-w_{i,j}$,
we are setting each $v_{i,j}^2=0$, and this
yields the 1-dimensional Cohen-Macaulay local ring
$$\frac{k[\![v_{1,1},v_{2,1},v_{1,2},v_{2,2},\ldots,v_{1,n},v_{2,n},v,w]\!]}
{\begin{pmatrix}v_{1,1}^2,v_{1,1}v_{2,1},v_{2,1}^2,v_{1,2}^2,v_{1,2}v_{2,2},v_{2,2}^2,\ldots,v_{1,n}^2,v_{1,n}v_{2,n},v_{2,n}^2,\\
vw,
vv_{1,1},vv_{2,1},vv_{1,2},vv_{2,2},\ldots,vv_{1,n},vv_{2,n}\end{pmatrix}}$$
which is isomorphic to the ring $R$.
\end{para}

We conclude this section with the proof of Theorem~\ref{thm20230825a}(\ref{prop171124b}).\vspace{4pt}


\noindent \emph{Proof of Theorem~\ref{thm20230825a}(\ref{prop171124b}).}
Continue with the terminology of~\ref{defn171101a}, and
let
$$k[G]'':=k[v_1,\ldots,v_n]/(I(G),v_1^2,\ldots,v_{n-1}^2).$$
Notice that $v_nv_i=0$ for all $i<n$ in $k[G]''$, but $v_n^2\neq 0$.
It then follows that
$$k[G]''\cong k[v_n]\times_k k[v_1,\ldots,v_{n-1}]/(I(H),v_1^2,\ldots,v_{n-1}^2).$$
Hence, by~\cite[Fact 2.2]{nasseh:ahplrdmi}, the ring $k[G]''$ has dimension $1$ and depth $0$.
Furthermore, polarizing shows that this ring deforms to the ring $k[\![\wti G]\!]$ with dimension $n$ and depth $n-1$, as desired.
\qed

\section{Gorenstein and complete intersection properties}\label{sec161010b}

As we mention in Propositions~\ref{cor161010b} and~\ref{cor161010by} below, the structure of Gorenstein local rings with decomposable maximal ideal (i.e., fiber product rings) can be described completely. Theorem~\ref{cor161010bz} generalizes this description to the local rings with quasi-decomposable maximal ideal. In this section, we prove this theorem and provide examples to show that local rings which are somewhat similar can be distinguished by the property of having quasi-decomposable maximal ideal or not. This persuades us to consider a relaxed version of the quasi-decomposable maximal ideal condition in Theorem~\ref{cor161010bz} and prove Theorem~\ref{prop171118a} as a more general version.

\begin{prop}[\protect{\cite[Corollary 2.7 and Fact 2.9]{nasseh:ahplrdmi}}]\label{cor161010b}
If $\fm_R$ is decomposable, then
$R$ is Gorenstein if and only if it is a $1$-dimensional hypersurface. In this case, if $R\cong S\times_k T$, then both $S$ and $T$ are $1$-dimensional regular local rings and $e(R)=2$.
\end{prop}

As an immediate consequence of this proposition we have the following result.

\begin{cor}\label{cor20220511a}
Let $R$ be a $1$-dimensional Gorenstein ring. Then, $\fm_R$ is quasi-decomposable if and only if it is decomposable.
\end{cor}

In Proposition~\ref{cor161010b}, if we assume that the local ring $R$ is complete, then we obtain more details on its structure.  

\begin{prop}\label{cor161010by}
Assume that $\fm_R$ is decomposable, and let $R\cong S\times_k T$. If $R$ is Gorenstein and complete,
then there exists a $2$-dimensional complete regular local ring $Q$ with regular system of parameters $r,s$ such that
$$S\cong Q/rQ,\qquad T\cong Q/sQ,\qquad \text{and}\qquad
R\cong Q/rsQ.$$
\end{prop}

\begin{proof}
The existence of the ring $Q$ with the desired properties comes from
Cohen's structure theorem, as in~\cite[Corollary~3.2.5]{takahashi:dssmrcf}.
\end{proof}

We now prove Theorem~\ref{cor161010bz} which is a generalization of Proposition~\ref{cor161010b}.\vspace{4pt}

\noindent \emph{Proof of Theorem~\ref{cor161010bz}.}
If $R$ is a hypersurface, then it is Gorenstein.

For the converse, assume that $R$ is Gorenstein. It follows that there is an $R$-regular sequence $\x=x_1,\ldots,x_c\in\m_R$ such that the maximal ideal $\m_{\ol R}$ of the ring $\ol R=R/(\x)$ is decomposable. Since $\ol R$ is also Gorenstein, Proposition~\ref{cor161010b} implies that $\ol R$ is a
$1$-dimensional hypersurface. Write $\ol R\cong S\times_k T$, where $S$ and $T$ are $1$-dimensional regular local rings.
Since $\m_{\ol R}=\m_S\oplus\m_T$,
it follows readily that $\edim(\ol R)=2$.
By construction, we have $\dim(R)=\dim(\ol R)+c=1+c$ and $\edim(R)\leq\edim(\ol R)+c=2+c$.
Hence,
$\edim(R)-\dim(R)\leq
(2+c)-(1+c)=
1$,
so $R$ is a hypersurface.
For the inequality involving $e(R)$, note that 
$e(R)\leq e(\ol R)=2$ by Proposition~\ref{cor161010b}.
\qed

\begin{para}\label{disc20230522a}
In contrast to Proposition~\ref{disc171118a}, if $R$ is a
singular $n$-dimensional hypersurface, then $\fm_R$ may or may not be quasi-decomposable.
For instance, Theorem~\ref{cor161010bz} rules out artinian hypersurfaces and the hypersurfaces of
multiplicity greater than 2.
However, even the hypersurfaces of dimension $1$ and multiplicity $2$ need not have quasi-decomposable maximal ideal.
Indeed, by Corollary~\ref{cor20220511a}, if $R$ is a 1-dimensional hypersurface that has quasi-decomposable maximal ideal, then it is not an integral domain. Hence, for any field $k$, the ring
$k[\![x,y]\!]/(x^2-y^3)\cong k[\![t^2,t^3]\!]$ does not have quasi-decomposable maximal ideal.

On the other hand, in higher dimensions, some integral domain hypersurfaces of multiplicity $2$ do have quasi-decomposable maximal ideal while others do not; see Examples~\ref{ex20220511a} and~\ref{ex20220511b} below.
\end{para}

\begin{ex}\label{ex20220511a}
Let $R=\bbc[\![x,y,z]\!]/(x^2+y^2+z^2)$. This ring is a $2$-dimensional
hypersurface domain that has quasi-decomposable maximal ideal. In fact,
the element $z$ is $R$-regular and we have
\begin{align*}
\frac{R}{zR}
&\cong\frac{\bbc[\![x,y]\!]}{(x^2+y^2)}
=\frac{\bbc[\![x,y]\!]}{(x+iy)(x-iy)}
\cong\frac{\bbc[\![u,v]\!]}{(uv)}
\cong\bbc[\![u]\!]\times_\bbc\bbc[\![v]\!].
\end{align*}
\end{ex}

\begin{ex}\label{ex20220511b}
If $g$ is an element of the cube of the maximal ideal $(x,y,z)$ of the ring $k[\![x,y,z]\!]$ over the field $k$,
then the hypersurface
$R=k[\![x,y,z]\!]/(x^2+g)$ does not have quasi-decomposable maximal ideal.
(Note that $e(R)=2$, and there are plenty of examples where this ring is an integral domain.)
By way of contradiction, suppose that $f\in k[\![x,y,z]\!]$ such that $\overline{f}\in R$ is $R$-regular and that $\ol R:=R/fR$ is isomorphic to a fiber product ring of the form $S\times_k T$.

Note that it follows readily that the maximal ideal of $\ol R/\m_{\ol R}^3$ is decomposable. Indeed, if $\m_{\ol R}=\m_S\oplus\m_T$, then
$\m_{\ol R}^3=\m_S^3\oplus\m_T^3$, so
the maximal ideal of $\ol R/\m_{\ol R}^3$ is
$$\m_{\ol R}/\m_{\ol R}^3=(\m_S\oplus\m_T)/(\m_S^3\oplus\m_T^3)\cong (\m_S/\m_S^3)\oplus(\m_T/\m_T^3).
$$
Note that the condition $g\in (x,y,z)^3$ implies that
\begin{equation}\label{eq171120a}
\frac{\ol R}{\m_{\ol R}^3}\cong \frac{k[\![x,y,z]\!]}{(x^2+g,f)+(x,y,z)^3}=\frac{k[\![x,y,z]\!]}{(x^2,f)+(x,y,z)^3}.
\end{equation}
If $f\in(x,y,z)^2$, then $e(\ol R)\geq 4$, contradicting Proposition~\ref{cor161010b}.
Thus, we have $f\in(x,y,z)\ssm(x,y,z)^2$. Now we consider two cases.

Case 1: $f\equiv ax^2\pmod{(x,y,z)^3}$ for some $a\in k$.
In this case, \eqref{eq171120a} reads as
$$\frac{\ol R}{\m_{\ol R}^3}\cong \frac{k[\![x,y,z]\!]}{(x^2)+(x,y,z)^3}.
$$
If the maximal ideal of this ring is decomposable, then in particular there are two linearly independent linear forms
$\alpha=bx+cy+dz$ and $\alpha'=b'x+c'y+d'z$ such that $\alpha\alpha'=0$ in $\ol R/\m_{\ol R}^3$,
that is, $\alpha\alpha'\in (x^2)+(x,y,z)^3\subset k[\![x,y,z]\!]$.
It is straightforward to show that there are no such forms, a contradiction.

Case 2: $f\not\equiv ax^2\pmod{(x,y,z)^3}$ for all $a\in k$.
In this case, since $f$ is in $(x,y,z)\ssm(x,y,z)^2$, the elements $x,f$ form part of a regular system of parameters for the ring $k[\![x,y,z]\!]$.
Let $x,f,u$ be a regular system of parameters for $k[\![x,y,z]\!]$.
Then, \eqref{eq171120a} reads as
$$\frac{\ol R}{\m_{\ol R}^3}\cong\frac{k[\![x,f,u]\!]}{(x^2,f)+(x,f,u)^3}\cong\frac{k[\![x,u]\!]}{(x^2)+(x,u)^3}.
$$
As in Case 1, it is straightforward to use linear forms to show that the maximal idea of this ring is indecomposable, again, a contradiction.

Thus, $R$ does not have quasi-decomposable maximal ideal.
\end{ex}

Examples~\ref{ex20220511a} and~\ref{ex20220511b} are interesting in that they show that rings which are somewhat similar can be distinguished by the property
of having quasi-decomposable maximal ideal (or not).
These rings have many similar homological properties, both being hypersurfaces. This fact can be seen by observing that
each one is a deformation of the regular local ring $k[\![x,y,z]\!]$, which \emph{does} have quasi-decomposable maximal ideal.
With this in mind, it is natural to consider a relaxed version of the quasi-decomposable maximal ideal condition.
This is explored in Theorem~\ref{prop171118a} whose proof is given next and concludes this section.\vspace{4pt}

\noindent \emph{Proof of Theorem~\ref{prop171118a}.}
\eqref{prop171118a1}$\implies$\eqref{prop171118a2}
Assume that $R$ is a complete intersection.
Let $\x\in \m_R$ be a maximal $R$-regular sequence, and set $R':=R/(\x)$. Note that $R'$ is artinian and hence, it is complete. By Cohen's Structure Theorem, $R'$ is a homomorphic image of a regular local ring $R''$ that can be chosen such that $\dim R''\geq 2$.
Since $R$ is a complete intersection, the same is true of $R'$, so the map $R''\to R'$ is a deformation.
Furthermore, $R''$ has quasi-decomposable maximal ideal by Proposition~\ref{disc171118a}.

\eqref{prop171118a2}$\implies$\eqref{prop171118a4}
Assume that $R$ is Gorenstein and admits a diagram of deformations
$R\to R'\from R''$ such that $R''$ has quasi-decomposable maximal ideal.
The deformation $R\to R'$ is quasi-Gorenstein by~\ref{para20230806a} and has finite complete intersection dimension.
Since $R''$ has quasi-decomposable maximal ideal, there is a deformation $R''\to R_2$ such that
$R_2$ has decomposable maximal ideal.
Now,
take the given diagram $R\to R'\from R''$ and set $\ti R=R'$ and $R_1=R''$ with $n=1$
to conclude that condition~\eqref{prop171118a4} holds.

\eqref{prop171118a4}$\implies$\eqref{prop171118a5}
follows from the facts that every deformation is a complete intersection local ring homomorphism and
every ring with decomposable maximal ideal has quasi-decomposable maximal ideal.

\eqref{prop171118a5}$\implies$\eqref{prop171118a1}
Under the assumptions, note that by~\ref{para20230806a} the ring $\ti R$ is Gorenstein since $R$ is Gorenstein and $R\to\ti R$ is quasi-Gorenstein.
The sequence~\eqref{eq20230519a} of complete intersection local ring homomorphisms shows that each ring $R_i$ is Gorenstein, by~\ref{para20230806n} and~\ref{para20230806a}.
It follows from Theorem~\ref{cor161010bz} that $R_n$ is a hypersurface,
so by~\ref{defn20230805e} each $R_i$ is a complete intersection. In particular, $\ti R$ is a complete intersection and hence, $R$ is a complete intersection by~\ref{defn171120f}.

\eqref{prop171118a1}$\implies$\eqref{prop171118a3}
If $R$ is a complete intersection, then there is a deformation $Q\to \comp R$, where
$Q$ is a regular local ring that can be chosen to have dimension $\geq 2$. Note that
$Q$ has quasi-decomposable maximal ideal by Proposition~\ref{disc171118a}.

\eqref{prop171118a3}$\implies$\eqref{prop171118a4}
Assume that $R$ is Gorenstein and the completion
$\comp R$ admits a deformation $\comp R\from Q$ such that $Q$ has quasi-decomposable maximal ideal.
Argue as in the proof of \eqref{prop171118a2}$\implies$\eqref{prop171118a4}, using the fact that the natural map $R\to \comp R$ is quasi-Gorenstein and has finite complete intersection dimension,
to conclude that condition~\eqref{prop171118a4} holds.
\qed

\section{Gorenstein property and the vanishing of Ext}\label{sec20220511a}
Our goal in this section is to prove Theorem~\ref{prop171123c} in which we study Tachikawa's Conjecture along certain diagrams of local ring homomorphisms that involve local rings with quasi-decomposable maximal ideal. We start with the next result which is essentially a consequence of~\cite[Corollary~6.3]{nasseh:lrqdmi}.

\begin{prop}\label{cor170106b}
If $\fm_R$ is decomposable, then the following are equivalent.
\begin{enumerate}[\rm(i)]
\item\label{cor170106b2}
$R$ is Gorenstein.
\item\label{cor170106b2'}
$R$ is a hypersurface.
\item\label{cor170106b1}
There exists
a non-zero finitely generated $R$-module $N$ with $\id_R(N)<\infty$ such that $\Ext^{i}_R(N,R)=0=\Ext^{i+1}(N,R)$ for
some $i\geq 5$.
\item\label{cor170106b3}
There exists a finitely generated $R$-module $M$ with $\pd_R(M)=\infty$ such that
$\Ext^{i}_R(M,R)=0=\Ext^{i+1}(M,R)$ for
some $i\geq 5$.
\end{enumerate}
\end{prop}

\begin{proof}
\eqref{cor170106b1}$\implies$\eqref{cor170106b2}
Under the assumptions of \eqref{cor170106b1}, it follows from~\cite[Corollary~6.3]{nasseh:lrqdmi} that $\id_R(R)<\infty$ or $\pd_R(N)<\infty$.
In the first case, $R$ is Gorenstein by definition.
In the latter case,
$R$ is  Gorenstein by a result of Foxby~\cite{foxby}.

\eqref{cor170106b2}$\implies$\eqref{cor170106b1}
If $R$ is Gorenstein, then the $R$-module $N=R$ satisfies   condition~\eqref{cor170106b1}.

\eqref{cor170106b3}$\implies$\eqref{cor170106b2}
follows again from~\cite[Corollary~6.3]{nasseh:lrqdmi}.

\eqref{cor170106b2}$\implies$\eqref{cor170106b3}
Since $R$ is a fiber product ring, it is not regular; see~\ref{para20230817a}.
Therefore, $M=k$ satisfies condition~\eqref{cor170106b3} since $R$ is Gorenstein.
%
%

Finally, \eqref{cor170106b2} $\Longleftrightarrow$ \eqref{cor170106b2'} follows from Proposition~\ref{cor161010b}.
\end{proof}

Proposition~\ref{cor170106b} can be souped up to the following result by a similar argument using~\cite[Corollary~6.8]{nasseh:lrqdmi} and Theorem~\ref{cor161010bz}.

\begin{prop}\label{prop171123b}
If $R$ is singular and $\fm_R$ is quasi-decomposable, then the following conditions are equivalent.
\begin{enumerate}[\rm(i)]
\item\label{prop171123b3}
$R$ is Gorenstein.
\item\label{prop171123b3'}
$R$ is a hypersurface.
\item\label{prop171123b1}
There exists
a non-zero finitely generated $R$-module $N$ with $\id_R(N)<\infty$ such that $\Ext^{i}_R(N,R)=0$ for  $i\gg 0$.
\item\label{prop171123b2}
There exists a finitely generated $R$-module $M$ with $\pd_R(M)=\infty$ such that
$\Ext^{i}_R(M,R)=0$ for  $i\gg 0$.
\end{enumerate}
\end{prop}

We now give the proof of Theorem~\ref{prop171123c}.\vspace{4pt}

\noindent \emph{Proof of Theorem~\ref{prop171123c}.}
Assume that \eqref{prop171123c1} holds.
The fact that $\psi$ is a deformation implies that
\begin{equation}\label{eq171123a}
R'\simeq\shift^e\Rhom[S]{R'}{S}
\end{equation}
where $e$ is the codimension of $\psi$.
Also, if $S$ were regular, then $R'$ would be Gorenstein, implying that $R$ is Gorenstein as well.
Therefore, we assume without loss of generality that $S$ is singular.

Write $\vf$ as a composition $\vf=\vf_1\cdots\vf_n$ of flat maps and deformations.
Rewrite each deformation as a composition of codimension-1 deformations, if necessary, to assume without loss of generality that
each deformation has codimension~1.
\vspace{4pt}

\noindent \textbf{Claim}: $\gdim_R(N)<\infty$.
\vspace{4pt}

To prove the claim, we argue by induction on $n$. For the base case $n=0$, note that $\vf$ is the identity on $R=R'$.
By assumption, we have
$\Ext^{i}_{R}(N,R)=0$ for  $i\gg 0$, that is, the $R$-complex $\Rhom{N}{R}$ is homologically bounded.
It follows from~\eqref{eq171123a} that
\begin{align*}
\Rhom[R]{N}{R}
&\simeq\Rhom[R]{N}{\shift^e\Rhom[S]{R}{S}}\\
&\simeq\shift^e\Rhom[R]{N}{\Rhom[S]{R}{S}}\\
&\simeq\shift^e\Rhom[S]{R\lotimes_RN}{S}\\
&\simeq\shift^e\Rhom[S]{N}{S}.
\end{align*}
Hence, the $S$-complex $\Rhom[S]{N}{S}$ is homologically bounded.
In other words, we have $\Ext^{i}_{S}(N,S)=0$ for  $i\gg 0$.
From~\cite[Corollary~6.8]{nasseh:lrqdmi}, we have that $\pd_S(N)<\infty$ or $S$ is Gorenstein.
In either of these cases, we have $\gdim_S(N)<\infty$. It now follows from~\cite[(2.2.8) Theorem]{lars} that
$\gdim_{R}(N)<\infty$ as well.
\vspace{4pt}

For $n\geq 1$ we consider the following cases.
\vspace{4pt}

Case 1: $\vf_n\colon R\to R''$ is flat.
By flat base change, the $R''$-module $N'':=R''\otimes_RN$ satisfies
$\Ext^{i}_{R''}(N'',R'')=0$ for  $i\gg 0$.
If $R''$ is regular, then so is $R$, which is a contradiction.
Thus, the diagram $R''\to R'\from S$ satisfies the hypotheses of our induction step,
so we conclude that $\gdim_R(N)=\gdim_{R''}(N'')<\infty$; see, for instance, \cite[(4.1.4)]{avramov:rhafgd}.
\vspace{4pt}
%

Case 2: $\vf_n\colon R\to R''$ is a codimension-1 deformation. In this case $\vf_n$ is surjective with kernel generated by
an $R$-regular element $x$.
Let $N_1$ be a syzygy of $N$, which implies that $x$ is $N_1$-regular.
Dimension-shifting implies that $\Ext^i_R(N_1,R)=0$ for all $i\gg 0$.
It follows that we have $N_1'':=R''\otimes_RN_1\simeq R''\lotimes_RN_1$ and therefore,
\begin{align*}
\Rhom[R'']{N_1''}{R''}
&\simeq\Rhom[R'']{R''\lotimes_RN_1}{R''}\\
&\simeq\Rhom{N_1}{\Rhom[R'']{R''}{R''}}\\
&\simeq\Rhom{N_1}{R''}.
\end{align*}
Thus, we have $\Ext^i_{R''}(N_1'',R'')\cong\Ext^i_R(N_1,R'')$.
Because of the short exact sequence $0\to R\xra{x} R\to R''\to 0$, using the assumption $\Ext^i_R(N_1,R)=0$ for $i\gg 0$,
we conclude that
$\Ext^i_{R''}(N_1'',R'')\cong\Ext^i_R(N_1,R'')=0$ for $i\gg 0$.
It follows from the induction step that
$\gdim_R(N_1)=\gdim_{R''}(R''\lotimes_RN_1)=\gdim_{R''}(N_1'')<\infty$; see~\cite[(4.31) Corollary]{auslander:smt}.
Since $N_1$ is a syzygy of $N$, it follows that $\gdim_R(N)<\infty$.

This establishes the claim.
\vspace{4pt}

Since we now have $\gdim_R(N)<\infty$ and $\id_R(N)<\infty$, it follows from~\cite[Theorem 3.2]{henrik} that $R$ is Gorenstein.

To show that \eqref{prop171123c2} implies that $R$ is Gorenstein,
argue as above
to conclude that either $\pd_R(M)<\infty$ or $R$ is Gorenstein.
Since $M$ has infinite projective dimension by assumption, we conclude that $R$ is Gorentein.
\qed\vspace{4pt}

Next example shows that part~\eqref{prop171123c2}
of Theorem~\ref{prop171123c} cannot be weakened to
having a diagram $R\xra\vf R'\xla\psi S$ described in part~\eqref{prop171123c1}.

\begin{ex}\label{ex171123aa}
Consider the Cohen-Macaulay local rings $S=k[\![x,y,z]\!]/(x^2,xy,y^2)$ and $R=k[\![x,y,z]\!]/(x^2,xy,y^2,z^2)$.
Note that $S$ has quasi-decomposable maximal ideal because the $S$-regular sequence $z$ satisfies
$$S/(z)\cong k[\![x,y]\!]/(x^2,xy,y^2)\cong k[\![x]\!]/(x^2)\times_k k[\![y]\!]/(y^2).$$
Also, $R$ is not Gorenstein and the natural projection $S\to R$ is a codimension-1 deformation. On the other hand, the $R$-module $R/(z)$ has infinite projective dimension and is totally reflexive
so it has lots of Ext-vanishing with respect to $R$; see~\ref{para20230808r} below for the definition of totally reflexive.
\end{ex}


The following result is a slight variation on the implication ``\eqref{prop171123c2}$\implies$$R$ is Gorenstein''
of Theorem~\ref{prop171123c}.

\begin{prop}\label{prop171123d}
Assume that there exists a finitely generated $R$-module $M$ with $\cidim_R(M)=\infty$ such that
$\Ext^{i}_R(M,R)=0$ for  $i\gg 0$,
and there is a diagram of local ring homomorphisms $R\xra\vf R'\xla\psi S$ such that
$\vf$ is flat,
$\psi$ is a deformation,
and $S$ has quasi-decomposable maximal ideal. Then,
$R$ is Gorenstein.
\end{prop}

\begin{proof}
If $\cidim_R(M)=\infty$, then it follows from~\cite[(1.13) Proposition]{avramov:cid} that $\cidim_{R'}(R'\otimes_RM)=\infty$ as well.
Now, argue as in the proof of Theorem~\ref{prop171123c}.
\end{proof}



We have seen in Theorems~\ref{prop171118a} and~\ref{prop171123c} that rings which admit certain diagrams
of local ring homomorphisms with the ring appearing on the right having quasi-decomposable maximal ideal have restrictive homological properties somehow similar to the rings with quasi-decomposable maximal ideals.
We conclude this section with a result about G-regularity that is of a similar spirit.

\begin{para}\label{para20230808r}
Following~\cite{takahashi:grlr}, the ring $R$ is called \emph{G-regular} if the class of totally reflexive $R$-modules (i.e., finitely generated $R$-modules of Gorenstein dimension $0$) coincides with the class of free $R$-modules.

\end{para}

%


\begin{prop}\label{prop171123a}
Let $\vf\colon R\to S$ be a local ring homomorphism
that is a composition of flat local ring homomorphisms and complete intersection local ring homomorphisms.
Assume that $S$ is Cohen-Macaulay and
has quasi-decomposable maximal ideal.
If $R$ is not a complete intersection, then $S$ and $R$ are both G-regular.
\end{prop}

\begin{proof}
Note that our assumptions imply that $\fd_R(S)<\infty$.
Since $R$ is not a complete intersection, finite flat dimension descent implies that $S$ is also not a complete intersection; see~\ref{defn171120f}.
Theorem~\ref{cor161010bz} implies that $S$ is not Gorenstein.
Our assumption that $S$ is Cohen-Macaulay and
has quasi-decomposable maximal ideal implies that $S$ is G-regular by~\cite[Corollary~6.6]{nasseh:lrqdmi}.
The proof of \emph{loc.\ cit}.\ shows that if $A\to B$ is a deformation such that $B$ is Cohen-Macaulay  and G-regular, then $A$ is Cohen-Macaulay  and G-regular.
It is straightforward to show that the same implication holds when the map $A\to B$ is flat and local.
Thus, it follows that $R$ is Cohen-Macaulay  and G-regular.
\end{proof}

\begin{para}
In contrast with Theorem~\ref{prop171118a}, one cannot improve Proposition~\ref{prop171123a} to allow for a zig-zag of local ring homomorphisms.
In fact, Example~\ref{ex171123aa} shows that if $S$ is Cohen-Macaulay and
has quasi-decomposable maximal ideal,
$R$ is not a complete intersection, and
$R\xra = R\xla\tau S$ is a diagram of local ring homomorphisms,
where $\tau$ is a codimension-1 deformation, then one cannot conclude that $R$ is G-regular.
\end{para}

%
%
%

\section{Semidualizing complexes} \label{sec170206a}

The notion of Semidualizing modules was originally introduced by Foxby~\cite{foxby:gmarm} and rediscovered by several authors independently for different applications; see, for instance~\cite{avramov:rhafgd, golod:gdagpi, sather:bnsc, vasconcelos:dtmc, wakamatsu:mtse}. Special cases of such modules include
canonical modules over Cohen-Macaulay rings, a notion that was introduced by Grothendieck; for more details see~\cite{hartshorne:lc}.

Our goal in this section is to prove Theorem~\ref{cor170213a} in which we show that the cardinality of the set consisting of shift-isomorphism classes of semidualizing $R$-complexes is small under the existence of a certain diagram of local ring homomorphisms. This set, which is denoted by $\s(R)$ (see~\ref{para20230810a} below), is known to be a finite set by~\cite{nasseh:lrfsdc}. On the other hand, for every integer $n\geq 1$, by~\cite[Theorem B]{nasseh:ahplrdmi}, there exists a local ring $R$ with $\card(\s(R))=2^n$. Hence, in general, $\s(R)$ can be big.

Note that for a single ring $R$, if $R$ is a fiber product ring or more generally, if $\fm_R$ is quasi-decomposable, then by~\cite[Corollary~4.6]{nasseh:vetfp} and Proposition~\ref{prop171127a} below
we have that $\s(R)$ is small. More precisely, in these cases we have $\card(\s(R))\leq 2$.
Theorem~\ref{cor170213a}, in fact, shows how far we can push this result along a zig-zag diagram of local ring homomorphisms.

\begin{para}
A finitely generated $R$-module $C$ is \emph{semidualizing} if
one has $R\cong \Hom_R(C,C)$ and $\Ext_R^i(C,C)=0$ for all $i\geq 1$. For instance, the free $R$-module $R^1$ is semidualizing.
A \emph{dualizing module} $D$ for $R$ is a semidualizing module with $\id_R(D)<\infty$.\footnote{The notions of dualizing module and canonical module agree when $R$ is Cohen-Macaulay.}

Note that $R$ admits a dualizing module if and only if it is Cohen-Macaulay and a homomorphic image of a local Gorenstein ring.
\end{para}

More generally, we define the following notions.

\begin{para}
A homologically finite $R$-complex $C$ is \emph{semidualizing} if
the natural homothety morphism $\chi^R_C\colon R\to\Rhom CC$ is an isomorphism in $\catd(R)$.
A \emph{dualizing complex} is a
semidualizing complex of finite injective dimension, i.e., a semidualizing complex that is isomorphic in $\catd(R)$
to a bounded complex of injective $R$-modules.
\end{para}

\begin{para}\label{para20230810a}
The set of isomorphism (resp. shift-isomorphism) classes of semidualizing $R$-modules (resp. $R$-complexes) in $\catd(R)$ is denoted $\s_0(R)$ (resp. $\s(R)$).
Note that $\s_0(R)$ is naturally a subset of $\s(R)$ because every semidualizing $R$-module is a semdualizing $R$-complex concentrated in degree 0.
\end{para}

\begin{para}
Note that $R$ is Gorenstein if and only if the free $R$-module $R^1$ of rank $1$ is dualizing for $R$, and this is if and only if $R^1$ is the only semidualizing $R$-complex up to shift-isomorphism in $\catd(R)$.
By~\cite{hartshorne:rad} and~\cite{Kaw} we know that $R$ has a dualizing complex if and only if it is a homomorphic
image of a local Gorenstein ring.
\end{para}

\begin{para}\label{para170508e}
The map on $\s$ induced by base-change along a local ring homomorphism of finite flat dimension is 1-1; see~\cite[Theorems 4.5 and 4.9]{frankild:rrhffd}.
\end{para}

\begin{prop}\label{prop171127a}
If $\m_R$ is quasi-decomposable, then $\card(\s(R))\leq 2$. More precisely, $\s(R)$ consists of the free $R$-module $R^1$ and dualizing $R$-complex (if it exists).
\end{prop}

\begin{proof}
Let $\x\in\m_R$ be an $R$-regular sequence such that $\ol R:=R/\x R$ is a fiber product.
As we mentioned in~\ref{para170508e}, the map $\s(R)\to\s(\ol R)$ induced by base-change is 1-1.
By~\cite[Corollary~4.6]{nasseh:vetfp} we have $\card(\s(\ol R))\leq 2$. More precisely, $\s(\ol R)$ consists of the free $\ol R$-module $\ol R^1$ and dualizing $\ol R$-complex (if it exists).
Hence, if $C\in\s(R)$, then $\ol C:=\ol R\lotimes_RC$ is shift-isomorphic to $\ol R$ or
it is dualizing for $\ol R$, in case that $\ol R$ has a dualizing complex.
In the first case, $C\simeq R$ up to a shift by, e.g., the standard equality of Poincar\'e series
$P^{\ol R}_{\ol C}(t)=P^R_C(t)$; for this equality see, for instance, \cite[(1.5.3) Lemma]{avramov:rhafgd}.
In the second case, $C$ must be dualizing for $R$ by \cite[(5.1)~Theorem]{avramov:glh} since every deformation is a Gorenstein local homomorphism.
\end{proof}

In order to prove Theorem~\ref{cor170213a} as a generalization of Proposition~\ref{prop171127a}, we need some more preliminary results,
beginning with a useful lemma that one can possibly deduce from results in~\cite{FSWW}. (Here, $\len_R(M)$ denotes the length of an $R$-module $M$.)

\begin{lem}\label{lem170213a}
Let $\vf\colon R\to S$ be a flat local ring homomorphism, and assume that the induced map $k\to S/\m_R S$ is an isomorphism.
Then, the induced map $\comp\phi\colon\comp R\to \comp S$ is also an isomorphism.
In particular, if $R$ is complete, then $\vf$ is an isomorphism.
\end{lem}

The assumption that the induced map $k\to S/\m_R S$ is an isomorphism is equivalent to the following: $\m_RS=\m_S$ and $\im(\varphi)+\m_S=S$.

\begin{proof}
Since $\vf$ is flat, the Nagata Flatness Theorem~\cite[Exercise 22.1]{M} implies that
\begin{equation}\label{eq20230524a}
\len_S(S/\m_R^nS)=\len_R(R/\m_R^n)\len_S(S/\m_R S)=\len_R(R/\m_R^n)
\end{equation}
for all positive integers $n$. Using the condition $k\cong S/\m_R S$, one sees that every composition series
for $S/\m_R^n S$ over $S$ is also a composition series over $R$, hence, the equality $\len_S(S/\m_R^nS)=\len_R(S/\m_R^nS)$ holds. Therefore, by~\eqref{eq20230524a} we have
$\len_R(S/\m_R^nS)=\len_R(R/\m_R^n)$.

The induced map $R/\m_R^n\to S/\m_R^n S$ is flat and local for all integers $n\geq 1$.
In particular, this map is injective.
The previous paragraph therefore implies that this map is bijective.
Passing to the inverse limit, we conclude that the induced map $\comp R\to\comp S$ is an isomorphism.
(Note that since $\m_R S=\m_S$, the $\m_R$-adic completion of $S$ is the same as its $\m_S$-adic completion.)
In particular, if $R$ is complete, then the composition
$\comp R=R\to S\to\comp S$ is an isomorphism; since each map in the composition is flat and local (hence, injective)
it follows that they are all also surjective.
\end{proof}

The next result complements~\cite[Proposition~3.15]{sather:divisor}.

\begin{prop}\label{prop170206a}
Let $\vf\colon R\to (S,\m_S,l)$ be a complete intersection local ring homomorphism of finite flat dimension.
Assume that $R$ is complete and the induced map $k\to l$ is an isomorphism.
Then the induced maps $$\s(R)\to\s(S)\to\s(\comp S)$$ are bijective.
\end{prop}

\begin{proof}
As we note in~\ref{para170508e}, the induced maps $\s(R)\to\s(S)\to\s(\comp S)$ are injective, so we only need to prove
surjectivity.

Case 1: $S$ is complete and $\vf$ is flat with regular closed fibre.
Let $\y\in\m_S$ give a minimal generating sequence $\overline{\y}$ for the maximal ideal $\m_S/\m_R S$ of the regular local ring $S/\m_R S$.
Since $\vf$ is flat, the fact that $\overline{\y}$ is $S/\m_R S$-regular implies that $\y$ is $S$-regular.
Moreover, the induced map $R\to S/(\y)S$ is flat; see~\cite[Corollary to Theorem 22.5]{M}.
By construction, the maximal ideal of $S/(\y)S$ is $\m_R (S/(\y)S)$.
Thus, the map $R\to S/(\y)S$ satisfies the hypotheses of Lemma~\ref{lem170213a}.
Since $R$ is assumed to be complete, we deduce from Lemma~\ref{lem170213a} that the map
$R\to S/(\y)S$ is an isomorphism. In particular, the induced map
$\s(R)\to \s(S/(\y)S)$ is bijective.
This surjective map factors as
$\s(R)\to\s(S)\to\s(S/(\y)S)$.
Since these maps are also injective, as we have noted, it is straightforward to deduce that they are both surjective.
Since $S$ is assumed to be complete, the proof in this case is finished.

Case 2: the general case. As in the previous case, it suffices to show that the induced map
$\s(R)\to\s(\comp S)$ is surjective. So, assume without loss of generality that $S$ is complete.
Consider a Cohen factorization $R\xra{\dot\vf}R'\xra{\vf'} S$ of $\vf$.
Since $\vf$ is complete intersection of finite flat dimension, the map $\vf'$ is a deformation.
Since $R'$ is complete, by~\cite[Proposition~4.2]{frankild:sdcms} the map $\s(R')\to\s(S)$ is surjective.
Now, Case 1 implies that $\s(R)\to\s(R')$ is also surjective, so the composition $\s(R)\to\s(S)$ is surjective as well.
\end{proof}


\begin{lem}\label{lem171127a}
Let $\vf\colon R\to S$ be a local ring homomorphism, and let $X$ be a homologically finite $R$-complex.
Then the following conditions are equivalent.
\begin{enumerate}[\rm(i)]
\item \label{lem171127a1}
$X$ is dualizing for $R$ and $\vf$ is quasi-Gorenstein.
\item \label{lem171127a2}
$X\in\s(R)$, $S\lotimes_RX$ is dualizing for $S$, and $\gdim(\vf)<\infty$.
\end{enumerate}
\end{lem}

\begin{proof}
\eqref{lem171127a1}$\implies$\eqref{lem171127a2}
follows from~\cite[(7.8)~Theorem]{avramov:rhafgd}.
%

\eqref{lem171127a2}$\implies$\eqref{lem171127a1}:
We have a series of equalities
\begin{align*}
P^R_{X}(t)I^X_R(t)I_{\vf}(t)
&=I^R_R(t)I_{\vf}(t)\\
&=I^S_S(t)\\
&=I_S^{S\lotimes_RX}(t)P^S_{S\lotimes_RX}(t)\\
&=t^aP^S_{S\lotimes_RX}(t)\\
&=t^aP^R_{X}(t)
\end{align*}
where, the first and third equalities are from~\cite[(3.18.2)]{christensen:scatac}, the second equality is~\eqref{eq20230811a}, the fourth equality
(for some $a\in\bbz$)
follows from the assumption
that $S\lotimes_RX$ is dualizing for $S$ and~\cite[V.3.4]{hartshorne:rad}, and the last equality is from~\cite[(1.5.3) Lemma]{avramov:rhafgd}. Cancellation implies that
$I^X_R(t)=t^b$ and $I_{\vf}(t)=t^c$ for some $b,c\in\bbz$;
the first of these equalities implies that $X$ is dualizing for $R$ again by~\cite[V.3.4]{hartshorne:rad}, and the second one implies that $\vf$ is quasi-Gorenstein by~\ref{para20230806a}.
\end{proof}

We are now ready to prove Theorem~\ref{cor170213a}.\vspace{4pt}

\noindent \emph{Proof of Theorem~\ref{cor170213a}.}
It suffices to show that if $C$ is a semidualizing $R$-complex such that $C\not\simeq\shift^iR$ for all $i\in\bbz$,
then $C$ is dualizing for $R$. Assume that such a $C$ is given.
We can assume without loss of generality that each $R_i$ is complete.
Note that this does not change the properties of the maps in the diagram nor the assumption about $R_n$.
We now argue by induction on $n\geq 0$.

The base case $n=0$ has been covered by Proposition~\ref{prop171127a}.

For the induction step, let $n\geq 2$, noting that the shape of the given diagram implies that $n$ is even.
Since the ring homomorphism $R_0\to R_1$ has finite complete intersection dimension, it follows from~\cite[(5.1)~Theorem]{christensen:scatac} and~\cite[Theorem~6.1(a)]{sather:cidfc} that $C_1:=R_1\lotimes_RC\in\s(R_1)$.
Then, Proposition~\ref{prop170206a} implies that there is a semidualizing $R_2$-complex $C_2$ such that
$C_1\simeq R_1\lotimes_{R_2}C_2$.
The standard equality of Poincar\'e series
$P^R_C(t)=P^{R_1}_{C_1}(t)=P^{R_2}_{C_2}(t)$ implies that for all $i\in\bbz$ we have $C_2\not\simeq\shift^iR_2$.
By our induction hypothesis, we conclude that $C_2$ is dualizing for $R_2$.
The fact that the local ring homomorphism $R_2\to R_1$ is complete intersection of finite flat dimension implies that it is Gorenstein. Therefore, \cite[(5.1)~Theorem]{avramov:glh}
implies that $C_1$ is dualizing for $R_1$.
Since the map $R_0\to R_1$ has finite complete intersection dimension, it has finite Gorenstein dimension. Hence, an application of Lemma~\ref{lem171127a}
shows that $C$ is dualizing for $R$, as desired.
%
%
\qed\vspace{4pt}

In light of the conclusions of Theorem~\ref{cor170213a}, it is clear that the hypotheses are restrictive.
The next result is another indication of this.

\begin{cor}\label{cor171127a}
Under the assumptions of Theorem~\ref{cor170213a}, either the ring $R$ is Gorenstein or each ring homomorphism $\vf_i\colon R_{2i}\to R_{2i+1}$ with $0\leq i\leq (n-2)/2$ is quasi-Gorenstein.
\end{cor}

\begin{proof}
As in the proof of Theorem~\ref{cor170213a}, we can assume without loss of generality that each $R_i$ is complete.
In particular, each $R_i$ has a dualizing complex $D_i$.
Assume that $R=R_0$ is not Gorenstein, so $D_0\not\simeq \shift^iR_0$ for all $i\in\bbz$.
We show by induction on $n$ that each map $\vf_i$ is quasi-Gorenstein.

The base case $n=0$ holds vacuously.
For the induction step, let $n\geq 2$, noting that the shape of the given diagram implies that $n$ is even.
Since $R_1$ is complete, by~\cite[(5.3)]{avramov:rhafgd} (see also~\ref{para20230808d}) the ring homomorphism $\vf_0$ has a dualizing complex $D^{\vf_0}$.
By definition, this means that
$D^{\vf_0}$ is a semidualizing $R_1$-complex such that
\begin{equation}\label{eq171127a}
D_1\simeq (D_0\lotimes_RR_1)\lotimes_{R_1}D^{\vf_0}\simeq D_0\lotimes_RD^{\vf_0}.
\end{equation}
Since $R_1$ is complete, $R_1\from R_2$ has finite flat dimension because it is a complete intersection ring homomorphism. On the other hand, by Theorem~\ref{cor170213a} we have $\card(\s(R_2))\leq 2$. Hence, by Proposition~\ref{prop170206a} we have $D^{\vf_0}\simeq\shift^iR_1$ or $D^{\vf_0}\simeq\shift^iD_1$ for some integer $i$.
If $D^{\vf_0}\simeq\shift^iD_1$, then taking Poincar\'e series in~\eqref{eq171127a} we have
\begin{align*}
P^{R_1}_{D_1}(t)
=P^{R}_{D_0}(t)P^{R_1}_{D^{\varphi_0}}(t)
=t^iP^{R}_{D_0}(t)P^{R_1}_{D_1}(t)
\end{align*}
where the left equality comes from~\cite[(1.5.3) Lemma]{avramov:rhafgd}. It follows that $P^{R}_{D_0}(t)=t^{-i}$, and therefore, $D_0\simeq \shift^iR_0$, which is a contradiction.

Hence, we must have $D^{\vf_0}\simeq\shift^iR_1$  for some $i\in\bbz$.
In other words, $R_1$ is a dualizing complex for $\vf_0$. Thus, $\vf_0$ is
quasi-Gorenstein by~\cite[(7.8)~Theorem]{avramov:rhafgd}.

For our induction argument, it remains to show that $R_2$ is not Gorenstein.
To this end, suppose by way of contradiction that $R_2$ were Gorenstein.
Then, Proposition~\ref{prop170206a} implies that $1=\card(\s(R_2))=\card(\s(R_1))$, i.e., $R_1$ is Gorenstein.
It follows that $R$ is Gorenstein by~\cite[(7.7.2)]{avramov:rhafgd}, which is a contradiction.
%
%
%
\end{proof}


\providecommand{\bysame}{\leavevmode\hbox to3em{\hrulefill}\thinspace}
\providecommand{\MR}{\relax\ifhmode\unskip\space\fi MR }
\providecommand{\MRhref}[2]{%
  \href{http://www.ams.org/mathscinet-getitem?mr=#1}{#2}
}
\providecommand{\href}[2]{#2}

\end{document}